\documentclass[12pt,a4paper]{amsart}

\usepackage{amssymb}
\usepackage{amscd}
\usepackage{hyperref}

\usepackage[numeric, abbrev, nobysame]{amsrefs}

\def\frak{\mathfrak}
\def\Bbb{\mathbb}
\def\Cal{\mathcal}

\let\phi\varphi

\newcommand{\x}{\times}
\renewcommand{\o}{\circ}

\newcommand{\al}{\alpha}
\newcommand{\be}{\beta}

\newcommand{\ep}{\epsilon}

\newcommand{\la}{\lambda}
\newcommand{\om}{\omega}
\newcommand{\ph}{\phi}
\newcommand{\ps}{\psi}

\newcommand{\si}{\sigma}

\newcommand{\Ga}{\Gamma}
\newcommand{\La}{\Lambda}
\newcommand{\Ph}{\Phi}

\newcommand{\Om}{\Omega}

\newcommand{\im}{\operatorname{im}}
\newcommand{\id}{\operatorname{id}}
\newcommand{\pr}{\operatorname{pr}}

\newtheorem*{thm*}{Theorem \thesubsection}

\newtheorem*{lemma*}{Lemma \thesubsection}
\newtheorem*{prop*}{Proposition \thesubsection}
\newtheorem*{cor*}{Corollary \thesubsection}

\theoremstyle{definition}

\newtheorem*{definition*}{Definition \thesubsection}

\newtheorem*{example*}{Example \thesubsection}
\theoremstyle{remark}

\newtheorem*{remark*}{Remark \thesubsection}

\def\sideremark#1{\ifvmode\leavevmode\fi\vadjust{\vbox to0pt{\vss
 \hbox to 0pt{\hskip\hsize\hskip1em
 \vbox{\hsize3cm\tiny\raggedright\pretolerance10000
  \noindent #1\hfill}\hss}\vbox to8pt{\vfil}\vss}}}%
                        
\begin{document}

\title{Pushing down the Rumin complex\\ 
to conformally symplectic quotients} 
\date\today
\author{Andreas \v Cap and Tom\'a\v s Sala\v c}
\thanks{Both authors gratefully acknowledge support by project
  P23244-N13 of the ``Fonds zur F\"orderung der wissenschaftlichen
  For\-schung'' (FWF). \newline
We also want to thank Michael Eastwood, who (after the first version
of this article was submitted) pointed out his article \cite{Eastwood}
to us.}

\address{Faculty of Mathematics\\
University of Vienna\\
Oskar--Morgenstern--Platz 1\\
1090 Wien\\
Austria}
\email{Andreas.Cap@univie.ac.at}
\email{Tomas.Salac@univie.ac.at}


\begin{abstract}
 Given a contact manifold $M_\#$ together with a transversal
 infinitesimal automorphism $\xi$, we show that any local leaf space
 $M$ for the foliation determined by $\xi$ naturally carries a
 conformally symplectic (cs--) structure. Then we show that the Rumin
 complex on $M_\#$ descends to a complex of differential operators on
 $M$, whose cohomology can be computed. Applying this construction
 locally, one obtains a complex intrinsically associated to any
 manifold endowed with a cs--structure, which recovers the
 generalization of the so--called Rumin--Seshadri complex to the
 conformally symplectic setting. The cohomology of this more general
 complex can be computed using the push--down construction.
\end{abstract}

\maketitle

\pagestyle{myheadings} \markboth{\v Cap and Sala\v c}{Pushing down the
  Rumin complex}
\section{Introduction}\label{1}

This article is motivated by the results \cite{EG} of M.~Eastwood and
H.~Goldschmidt on integral geometry and the subsequent work \cite{ES}
of M.~Eastwood and J.~Slov\'ak on conformally Fedosov structures. The
main tool used in \cite{EG} is a family of complexes of differential
operators on complex projective space $\Bbb CP^n$. The results on
integral geometry are deduced from vanishing of some cohomology groups
of these complexes. The form and length of these complexes is rather
intriguing and the article \cite{ES} takes steps towards an
explanation. The main notion introduced there is the one of a
conformally Fedosov structure, which combines a conformally symplectic
structure and a projective structure, which satisfy a suitable
compatibility condition. Given these data, the authors construct a
tractor bundle endowed with a (linear) tractor connection which is
naturally associated to the conformally Fedosov structure. This should
open the possibility to construct sequences and complexes of
differential operators following the ideas of the
Bernstein--Gelfand--Gelfand (BGG) machinery as introduced in
\cite{CSS-BGG} and \cite{CD} in the setting of parabolic geometries.

The tractor bundle associated to a conformally Fedosov structure looks
similar to the standard tractor bundle associated to a contact
projective structure (see \cite{Fox}). This is an instance of a
so--called \textit{parabolic contact structure}, the best know example
of which are (hypersurface type) CR structures. It is known that the
homogeneous models of parabolic contact structure are, via forming
quotients by transversal infinitesimal automorphisms, related to
special symplectic connections, see \cite{CS} and Sections 5.2.18 and
5.2.19 of \cite{book} for an exposition in the language of parabolic
geometries.

The starting point for our considerations is the hope to obtain
complexes like the ones constructed in \cite{EG} from BGG sequences
associated to parabolic contact structures via similar quotient
constructions. In this article, we show that this can indeed be done
in the special case of the BGG sequence associated to the trivial
representation. It is shown in \cite{BEGN} that in this case one
obtains the Rumin complex (see \cite{Ru}), which can be naturally
constructed for any contact structure, see \cite{BEGN} for a simple
direct construction. From either construction it follows that the
Rumin complex is a fine resolution of the constant sheaf $\Bbb R$, so
in particular it computes the de--Rham cohomology of a contact
manifold. Following this, our article also works in the setting of
general contact manifolds and does not use parabolic geometry
techniques.

We first prove that the quotient of a contact structure by a
transversal infinitesimal automorphism naturally inherits a symplectic
structure, and thus in particular a conformally symplectic structure
(or cs--structure). In contrast to the traditional approach to
defining such a structure via a specific two--form, we just view it as
an appropriate line subbundle in the bundle of two--forms, which
simplifies matters in several respects.

Next, we show that the Rumin complex can be pushed down to a complex
of differential operators on the quotient space, which coincides with
the complex on a symplectic manifold constructed in \cite{TY1,TY2} and
in \cite{BEGN}, where it is called \textit{Rumin--Seshadri} complex,
see also \cite{RuS}. The push--down construction easily leads to a
long exact sequence relating the cohomology of this complex to
de--Rham cohomology. In particular, one can immediately read off (in
this very simple special case) the cohomological information needed in
the applications in \cite{EG}.

To complete the picture, we prove that the push down construction can
be used to construct a version of this complex and an analog of the
long exact sequence on any smooth manifold endowed with a
cs--structure. We prove that any such manifold can be locally
represented as the quotient of a contact structure by a transversal
infinitesimal automorphism. Moreover, any local isomorphism of
cs--structures can be lifted to a contactomorphism
``upstairs''. Naturality of the Rumin complex under contactomorphisms
then implies that one can use the local contactifications to obtain a
complex and a long exact sequence on the whole cs--manifold and that
they are intrinsic to the cs--structure.

The extension of the construction of the Rumin--Seshadri complex from
symplectic manifolds to cs--manifolds is not a new result in its own
right, a direct construction is available in \cite{Eastwood}. The main
advantage of our approach is not the result itself, but the potential
for generalizations.

\section{A quotient of the Rumin complex}\label{2} 

We start by discussing local quotients of contact manifolds by
transverse infinitesimal automorphisms. 

\subsection{Contact manifolds and differential forms}\label{2.1}
By a contact manifold, we will always mean a manifold $M_\#$ of odd
dimension $2n+1$ endowed with a maximally non--integrable distribution
$H\subset TM$ of rank $2n$. While this implies that locally $H$ can be
written as the kernel of a one--form $\al\in\Om^1(M)$ such that
$\al\wedge(d\al)^n$ is nowhere vanishing, we do not initially assume
that such a \textit{contact form} exists globally or that there
locally is a preferred choice.

The best way to conceptually formulate the condition of maximal
non--in\-te\-gra\-bi\-li\-ty is via the \textit{Levi--bracket}. Defining
$Q:=TM_\#/H$, which clearly is a line bundle naturally associated to
any corank one subbundle in the tangent bundle, the Levi--bracket is
the bilinear bundle map $\Cal L:H\x H\to Q$ induced by the Lie bracket
of vector fields. The condition that $H$ defines a contact structure
is than exactly that $\Cal L$ is non--degenerate in each point. In
this case, $\Cal L$ induces an isomorphism $H\cong H^*\otimes Q$ of
vector bundles, whose inverse can be also viewed as a non--degenerate
bilinear bundle map $\Cal L^{-1}:H^*\x H^*\to Q^*$ respectively as an
element of $\La^2H\otimes Q^*$.

We can use the Levi--bracket and its inverse to construct subbundles
in the exterior powers of $H^*$. Indeed, wedging with $\Cal L$ can be
viewed as a bundle map $\La^kH^*\to \La^{k+2}H^*\otimes Q$, while
insertion of $\Cal L^{-1}$ defines a map $\La^kH^*\to
\La^{k-2}H^*\otimes Q^*$. Non--degeneracy implies that the first map
is injective for $k\leq n-1$ and surjective for $k\geq n-1$ (when
$\dim(M_\#)=2n+1$) while the second map is surjective for $k\leq n+1$
and injective for $k\geq n+1$. Thus we can define
$\La^k_0H^*\subset\La^kH^*$ for $k\leq n$ as the kernel of the
insertion of $\Cal L^{-1}$ and for $k>n$ as the kernel of wedging with
$\Cal L$.

Thus, a contact structure induces a filtration on the bundles of
differential forms, as well as a finer decomposition of the bundles
occurring as subquotients of this filtration. The passage from the
de--Rham complex to the Rumin complex can be viewed as a passage from
the bundles of differential forms to some of this subquotient bundles,
which does not change the cohomology. The price one has to pay for
restricting to simpler bundles, is that (at least in one instance) the
exterior derivative has to be replaced by a higher order operator. Let
us collect the information on differential forms on a contact manifold
we will need.

\begin{prop*}
Let $(M_\#,H)$ be a contact manifold and put $Q:=TM_\#/H$. Then for
each $k=1,\dots,2n$ we get an exact sequence 
$$
0\to \La^{k-1}H^*\otimes Q^*\to \La^kT^*M_\#\to\La^kH^*\to 0
$$ 
as well as decompositions
\begin{gather*}
\La^kH^*=\La^k_0H^*\oplus(\La^{k-2}_0H^*\otimes Q^*)\oplus\dots\oplus
(\La^{k-2i}_0H^*\otimes\otimes^i Q^*) \quad \text{\ for $k\leq n$}.\\
\La^kH^*=\La^k_0H^*\oplus (\La^{k+2}_0H^*\otimes Q)\oplus\dots\oplus
(\La^{k+2j}_0H^*\otimes\otimes^j Q) \qquad \text{\ for $k>n$}.
\end{gather*}
Here $i$ is the largest integer such that $2i\leq k$ and $j$ is the
largest integer such that $k+2j\leq 2n$. 
\end{prop*}
\begin{proof}
  The first statement follows immediately from dualizing the exact
  sequence defining $Q$ and then taking exterior powers. For the first
  decomposition, we use insertion of $\Cal L^{-1}$ to obtain a
  surjection from $\La^kH^*$ onto $\La^{k-2}H^*\otimes Q^*$ with
  kernel $\La^k_0H^*$. The wedge product with $\Cal L$ defines a
  splitting of the corresponding exact sequence, which shows that
  $\La^kH^*\cong \La^k_0H^*\oplus(\La^{k-2}H^*\otimes Q^*)$. From this
  the statement follows by induction. For the second decomposition,
  one argues in the same way exchanging the roles of the two maps.
\end{proof}

\subsection{Cs--quotients of contact manifolds}\label{2.2}

\begin{definition*}
  Let $(M_\#,H)$ be a contact manifold. A \textit{transversal
    infinitesimal automorphism} of $M$ is a vector field $\xi\in\frak
  X(M)$ such that $\xi(x)\notin H_x$ for all $x\in M_\#$ and such that
  the local flow of $\xi$ preserves the contact distribution.
\end{definition*}

Note that in particular $\xi$ is nowhere vanishing on $M_\#$, thus
defining a foliation of $M_\#$ with one--dimensional leaves. On the
other hand, the condition that the flow of $\xi$ preserves the contact
distribution is easily seen to be equivalent to the fact that for any
$\eta\in\Ga(H)\subset\frak X(M_\#)$ the Lie bracket $[\xi,\eta]$ also
lies in $\Ga(H)$.  

Next, we define a \textit{quotient} of a contact manifold $M_\#$ by a
transversal infinitesimal automorphism $\xi\in\frak X(M_\#)$ to be a
surjective submersion $q:M_\#\to M$ onto a smooth manifold $M$ of
dimension $2n$ such that for each $x\in M_\#$ the kernel of
$T_xq:T_xM_\#\to T_{q(x)}M$ is spanned by $\xi(x)$ and such that the
fibers of $q$ are connected. 

This simply means that $M$ is a (global) space of leaves for the
foliation defined by $\xi$. Local existence of such leaf spaces is a
consequence of the Frobenius theorem, and this is all we will formally
need in the sequel, since our results are local in nature. A global
notion is only used in order to be able to work with larger leaf space
than the ones provided by the Frobenius theorem in case that they are
available. First, we show that the quotient space naturally inherits a
symplectic structure. For later use it will be more useful to
emphasize the conformally symplectic aspects of the construction,
where we view a \textit{conformally symplectic} structure as a line
subbundle $\ell$ in the bundle of two--forms, whose non--zero elements
are all non--degenerate and which locally admits sections which are
closed as two--forms.

\begin{prop*}
  Let $q:M_\#\to M$ be a quotient of a contact manifold $(M_\#,H)$ by
  a transversal infinitesimal automorphism $\xi$. Then we have:

(1) For each $x\in M_\#$, the tangent map $T_xq$ restricts to a linear
  isomorphism $H_x\to T_{q(x)}M$. Via this isomorphism, $\Cal L$
  determines a line subbundle $\ell\subset\La^2T^*M$ such that each
  non--zero element of $\ell$ is non--degenerate as a bilinear form on
  the corresponding tangent space.

(2) There is a unique contact form $\al$ on $M_\#$ such that
  $\al(\xi)=1$ and a unique symplectic form $\om$ on $M$ such that
  $q^*\om=d\al$. This form $\om$ is a section of
  $\ell\subset\La^2T^*M$, so $\ell$ defines a conformally
  symplectic structure on $M$.
\end{prop*}
\begin{proof} 
(1) By definition, the kernel of $T_xq$ is the line spanned by
  $\xi(x)$, which has zero intersection with $H_x$ so the first claim
  follows. Fixing $x$, $\Cal L_x:\La^2H_x\to Q_x$ can be viewed as a
  non--degenerate bilinear form on $H_x$ determined up to scale. Via
  the isomorphism $H_x\to T_{q(x)}M$, it thus determines a line in
  $\La^2T_{q(x)}^*M$ whose non--zero elements are
  non--degenerate. Since $\xi$ is an infinitesimal automorphism, its
  flow preserves the contact distribution. Hence the tangent maps of
  its flow induce bundle maps $H\to H$ and $Q\to Q$. Naturality of the
  Levi--bracket implies that these bundle maps are compatible with
  $\Cal L$, and since the flow lines of $\xi$ are the fibers of $q$,
  we conclude that the line in $\La^2T_{q(x)}^*M$ constructed above
  only depends on $q(x)$ and not on $x$. Choosing a local section of
  $q$, one easily concludes that this defines a smooth line subbundle
  $\ell\subset\La^2T^*M$.

(2) By definition, $\xi$ defines a nowhere--vanishing section of the
  bundle $Q=TM_\#/H$, thus trivializing this line bundle. Hence there
  is a unique section of $Q^*$ pairing to one with that section. Since $Q^*$
  can be identified with the annihilator of $H$ in $T^*M_\#$ this
  section can be viewed as a contact form $\al$ such that
  $\al(\xi)=1$. Of course, this uniquely determines $\al$ among all
  contact forms. Since $\xi$ is an infinitesimal automorphism, we must
  have $0=\al([\xi,\eta])=d\al(\xi,\eta)$ for all $\eta\in\Ga(H)$. Of
  course $d\al(\xi,\xi)=0$, so $i_\xi d\al=0$. This also shows that
  $L_\xi d\al=0$, where $L_\xi$ denotes the Lie derivative along
  $\xi$. By Corollary 2.3 in \cite{BCG3} the last two properties imply
  that there is a two--form $\om$ on $M$ such that $d\al=q^*\om$. In
  particular, this shows that $0=dq^*\om=q^*d\om$ and since $q$ is a
  surjective submersion, this implies that $\om$ is closed and
  uniquely determined.

  Finally, the restriction of $d\al$ to $\La^2H$ can be written as the
  composition of the trivialization of $Q^*$ defined by evaluating
  elements on $\xi$ with the Levi bracket. This implies that $\om$ is
  a section of $\ell$, which completes the proof.
\end{proof}

Having the cs--structure $\ell$ on $M$, we get decompositions of the
spaces of differential forms similarly to the ones in Proposition
\ref{2.1}. We first observe that the inclusion $\ell\hookrightarrow
\La^2T^*M$ can be viewed as a canonical section $\mathbf{\Om}$ of
$\La^2T^*M\otimes\ell^*$. Non--degeneracy implies that $\mathbf{\Om}$
defines an isomorphism $TM\to T^*M\otimes\ell^*$, whose inverse can be
viewed as a canonical section $\mathbf{\Om}^{-1}$ of the bundle
$\La^2TM\otimes\ell$. In particular, wedging by $\mathbf{\Om}$ defines
a bundle map $\La^kT^*M\to\La^{k+2}T^*M\otimes \ell^*$ while insertion of
$\mathbf{\Om}^{-1}$ induces a bundle map
$\La^kT^*M\to\La^{k-2}T^*M\otimes\ell$. As before, non--degeneracy
implies that the first map is injective for $k\leq n-1$ and surjective for
$k\geq n-1$ while the second map is surjective for $k\leq n+1$ and
injective for $k\geq n+1$. Similarly to \ref{2.1}, we define
$\La^k_0T^*M$ to be the kernel of the insertion of $\mathbf{\Om}^{-1}$
for $k\leq n$ and the kernel of wedging by $\mathbf{\Om}$ for $k>n$.

\begin{cor*}The conformally symplectic structure $\ell$ on $M$ gives
  rise to decompositions 
\begin{gather*}
\La^kT^*M=\La^k_0T^*M\oplus(\La^{k-2}_0T^*M\otimes
\ell)\oplus\dots\oplus (\La^{k-2i}_0T^*M\otimes\otimes^i\ell)
\quad\text{\ for $k\leq
  n$}.\\ \La^kT^*M=\La^k_0T^*M\oplus(\La^{k+2}_0T^*M\otimes
\ell^*)\oplus\dots\oplus (\La^{k+2j}_0T^*M\otimes\otimes^j\ell)
\quad\text{\ for $k>n$}.
\end{gather*}
Here $i$ is the largest integer such that $2i\leq k$ and $j$ is the
largest integer such that $k+2j\leq 2n$. 

Moreover, for each $r$, the linear maps induced by $T_xq$ restrict to
linear isomorphisms $\La^r_0H_x^*\cong \La^r_0T_{q(x)}^*M$.
\end{cor*}
\begin{proof}
The decompositions are proved exactly as in the proof of Proposition
\ref{2.1}. The last statement follows readily from the proof of
part (1) of Proposition \ref{2.2}.
\end{proof}

\subsection{Pushing down the Rumin complex}\label{2.3}
For any contact manifold $(M_\#,H)$, the filtration of the cotangent
bundle from Proposition \ref{2.3} makes the de--Rham complex
$(\Om^*(M_\#),d)$ into a ($2$--step) filtered complex. This is
determined by the subspaces $\Cal F_k\subset\Om^k(M_\#)$ formed by
those $k$--forms which vanish if all their entries are from $H\subset
TM_\#$. This means that $\Cal F_k=\Ga(\La^{k-1}H^*\otimes Q^*)$ and
$\Om^k(M_\#)/\Cal F_k\cong\Ga(\La^kH^*)$, with the last isomorphism
induced by restricting $k$--forms to multilinear maps defined on $H$. 

The simplest construction of the Rumin complex comes from the spectral
sequence associated to this filtration of the de--Rham complex, see
\cite{BEGN}. It is easy to see that restricting the exterior
derivative to $\Cal F_k$ and then composing with the projection
$\Om^{k+1}(M_\#)/\Cal F_{k+1}$ one obtains a tensorial operator
induced by the bundle map 
\begin{equation}\label{part}
\partial:\La^{k-1}H^*\otimes Q^*\to\La^{k+1}H^*,
\end{equation}
which (up to a non--zero factor) is just the composition of the
alternation with $\Cal L^*:Q^*\to\La^2H^*$ tensorized by
$\id_{\La^{k-1}H^*}$. Hence linear algebra implies that if
$\dim(M_\#)=2n+1$ the bundle map $\partial$ in \eqref{part} is
injective if $k\leq n$ and surjective if $k\geq n$. In particular,
from Proposition \ref{2.1} we see that the cohomology bundle $\Cal
H^k_\#:=\ker(\partial)/\im(\partial)$ of $\partial$ in degree $k$ is
$\La^k_0H^*$ for $k\leq n$ and $\La^{k-1}_0H^*\otimes Q^*$ for $k>n$.

Now the standard construction of spectral sequences provides operators
$D_k:\Ga(\Cal H^k_\#)\to\Ga(\Cal H^{k+1}_\#)$, which are easily seen to be
differential operators forming a complex. This is the Rumin complex,
which by construction computes the de--Rham cohomology of $M_\#$. 

\medskip

Now assume that we have given quotient $q:M_\#\to M$ of $M_\#$ by a
transversal infinitesimal automorphism $\xi\in\frak X(M_\#)$. Then we
know from \ref{2.2} that $M$ inherits a conformally symplectic
structure $\ell\subset\La^2T^*M$. The idea for pushing down the Rumin
complex to a complex of differential operators on $M$ is easy. One
just restricts to the subsheaves of forms in the kernel of the Lie
derivative $L_\xi$.

Observe first, that the fact that $\xi$ is an infinitesimal
automorphism of the contact manifold $(M_\#,H)$ immediately implies
that $L_\xi(\Cal F_k)\subset\Cal F_k$ for all $k$. We have noted
already that one obtains a well defined Lie derivative on the bundles
$H$ and $Q$ and thus on their duals and all bundles obtained by
tensorial constructions from them. Naturality of the Lie derivative
then implies that, for each $k$, we get a commutative diagram with
exact rows
$$
\begin{CD}
0 @>>> \Ga(\La^{k-1}H^*\otimes Q^*) @>>> \Om^k(M_\#) @>>>
\Ga(\La^kH^*) @>>> 0\\
@. @VL_\xi VV  @VL_\xi VV  @VL_\xi VV @.\\
 0 @>>> \Ga(\La^{k-1}H^*\otimes Q^*) @>>> \Om^k(M_\#) @>>>
\Ga(\La^kH^*) @>>> 0
\end{CD}
$$ 
Since $L_\xi$ commutes with the exterior derivative, we conclude from
the above construction that also $L_\xi\o\partial=\partial\o L_\xi$
and $D_k\o L_\xi=L_\xi\o D_k$. Hence we conclude that $D_k$ induces an
operator mapping $\ker(L_\xi)\subset\Ga(\Cal H^k_\#)$ to
$\ker(L_\xi)\subset\Ga(\Cal H^{k+1}_\#)$. Now we can determine
explicitly, what these kernels look like. For $k\geq 0$, we denote the
space of sections of $\La^k_0T^*M$ by $\Om^k_0(M)$.

\begin{prop*}
Let $q:M_\#\to M$ be a quotient of a contact manifold $(M_\#,H)$ by a
transversal infinitesimal automorphism $\xi\in\frak X(M)$. Then for
any $k$ we have the following isomorphisms
\begin{center}
\begin{tabular}{|l|l|l|}
\hline
Space & $\ker(L_\xi)$ isomorphic to & isomorphism \\
\hline
$\Ga(\La^kH^*)$ & $\Om^k(M)$ & $\ph\mapsto q^*\ph|_{\La^kH}$ \\
\hline
$\Ga(\La^k_0H^*)$ & $\Om^k_0(M)$ & $\ph\mapsto q^*\ph|_{\La^kH}$  \\
\hline
$\Ga(\La^kH^*\otimes Q^*)$ & $\Om^k(M)$ &
$\ph\mapsto\al\wedge q^*\ph$ \\
\hline
$\Ga(\La^k_0H^*\otimes Q^*)$ & $\Om^k_0(M)$ &
$\ph\mapsto\al\wedge q^*\ph$  \\
\hline
$\Om^k(M_\#)$ & $\Om^k(M)\oplus\Om^{k-1}(M)$ &
$(\ph_1,\ph_2)\mapsto q^*\ph_1+\al\wedge q^*\ph_2$  \\
\hline
\end{tabular}
\end{center}
In this table, $\al$ denotes the unique contact form on $M_\#$ such
that $\al(\xi)=1$ and the isomorphisms map the spaces in the second
column to the kernel of $L_\xi$ contained in the space in the first
column. 
\end{prop*}
\begin{proof}
The quickest way to prove this is via the well known characterization
of pullback forms on $M_\#$, see Corollary 2.3 in \cite{BCG3}, which
applies by connectedness of the fibers of $q$. In our situation this
says that a form $\ps\in\Om^k(M_\#)$ is of the form $q^*\ph$ for some
$\ph\in\Om^k(M)$ if and only if $i_\xi\ps=0$ and $L_\xi\ps=0$. Since
$q$ is a surjective submersion, the form $\ph$ is then uniquely
determined by $\ps$.

Starting with $\ps\in\Om^k(M_\#)$ such that $L_\xi\ps=0$ and taking
into account that $i_\xi$ and $L_\xi$ commute, we can use this to
conclude that $i_\xi\ps\in\Om^{k-1}(M_\#)$ equals $q^*\ph_2$ for a
unique $\ph_2\in\Om^k(M)$. Then $\ps-\al\wedge q^*\ph_2$ by
construction lies in the kernel of $i_\xi$ and since $L_\xi\al=0$, it
is of the form $q^*\ph_1$ for a unique form $\ph_1\in\Om^k(M)$. Since
the converse inclusion is obvious this establishes the last
isomorphism.

Next, given $\ph\in\Om^k(M)$ we can form $q^*\ph\in\Om^k(M_\#)$ and
project it to the quotient $\Ga(\La^kH^*)$ to obtain a section $\si$,
which by construction satisfies $L_\xi\si=0$. Conversely, given such a
section $\si$, we can extend it arbitrarily to a smooth form
$\tilde\ps\in\Om^k(M_\#)$ and then define $\ps:=\tilde\ps-\al\wedge
i_\xi\tilde\ps$. By construction, this is still an extension of $\si$
and $i_\xi\ps=0$. (Indeed, it is uniquely determined by these two
properties.) Moreover, we also get $i_\xi L_\xi\ps=0$. But since $\ps$
projects to $\si$, naturality of the Lie derivative implies that
$L_\xi\ps$ projects to $L_\xi\si=0$ in $\Ga(\La^kH^*)$. Thus we
conclude that $L_\xi\ps=0$, so $\ps=q^*\ph$ for some $\ph\in\Om^k(M)$
and the first isomorphism is established.

From the construction and the definitions it is clear, that in the
above we can restrict to $\ph\in\Ga(\La^k_0T^*M)$ and then the
restriction of $q^*\ph$ will have values in
$\Ga(\La^k_0H^*)$. Conversely, given $\si\in\Ga(\La^k_0H^*)$, the
above construction leads to $\ph\in\Ga(\La^k_0T^*M)$, so the second
isomorphism follows as above.

Finally, $\Ga(\La^kH^*\otimes Q^*)\subset\Om^{k+1}(M_\#)$ consists of
those forms which vanish if all their entries are from $H$. If $\ps$
is such a form, then $\ps=\al\wedge i_\xi\ps$. As above, $L_\xi\ps=0$
implies $L_\xi i_\xi\ps=0$, so $i_\xi\ps=q^*\ph$ for some
$\ph\in\Om^k(M)$. Since the converse inclusion is obvious, the third
isomorphism is proved. The passage to $\La^k_0H^*\otimes Q^*$ upstairs
and $\La^k_0T^*M$ downstairs needed to obtain the fourth isomorphism
is again straightforward.
\end{proof}

\begin{remark*}
In view of potential generalizations, we want to point out that the
first four isomorphisms claimed in the theorem can actually be proved
directly without passing through differential forms (and in more
general situations such proofs will be simpler). The fact that passing
through forms is more efficient is only due to the simplicity of the
filtration of the involved bundles.  
\end{remark*}

From this result it is clear that the Rumin complex on $M_\#$ descends
to a complex on $M$. To understand which cohomology this complex
computes, we only have to interpret what we have done in terms of
sheaf theory. Of course, we can view the Lie derivative $L_\xi$ as an
endomorphism of the sheaf of $k$--forms on $M_\#$ and thus its kernel
as a subsheaf $\Cal A^k$ of the sheaf of $k$--forms on $M_\#$. Via the
map $q:M_\#\to M$, one forms a sheaf $q_*\Cal A^k$ on $M$, which by
definition to an open subset $U$ of $M$ associates $\Cal
A^k(q^{-1}(U))$. Since the exterior derivative (on $M_\#$) is a
morphism $\Cal A^k$ to $\Cal A^{k+1}$, we obtain a complex of sheaves
$(q_*\Cal A^*,d)$ on $M$.

\begin{cor*}
Via the isomorphisms from Proposition \ref{2.3}, the Rumin complex on
$M_\#$ descends to a complex $(\Cal H^i,D_i)$ of differential
operators where $\Cal H^i=\Om^i_0(M)$ for $i\leq n$ and $\Cal
H^i=\Om^{i-1}_0(M)$ for $i>n$. The operators $D_i$ are all of first
order, except for $D_n:\Om^n_0(M)\to\Om^n_0(M)$ which has order two.

The cohomology of this complex coincides with the sheaf cohomology of
the complex $(q_*\Cal A^*,d)$ of sheaves on $M$.
\end{cor*}
\begin{proof}
The first part follows immediately from Proposition \ref{2.3} and the
corresponding results for the Rumin complex. For the last part, recall
that the Rumin complex arises from the de--Rham resolution on $M_\#$
via the spectral sequence associated to a (very simple)
filtration. This filtration restricts to $\Cal A^*$ and we can again
use the corresponding spectral sequence to compute the cohomology. But
Proposition \ref{2.3} shows that in the first step of this spectral
sequence one obtains sheaves that descend to $M$, and the operators on
these sheaves produced by the spectral sequence exactly form the
complex $(\Cal H^i,D_i)$.
\end{proof}

\subsection{Cohomology of the descended complex}\label{2.4}

To compute the cohomology of the descended complex, we have to
understand the complex $(q_*\Cal A^k,d)$ of sheaves. This again is
easy using Proposition \ref{2.3}. 

\begin{thm*}
For any open subset $U\subset M$ let us denote by $H^k(U)$ the $k$--th
de--Rham cohomology of $U$. Then the cohomology groups of the
complex $(\Cal H^*|_U,D)$ over $U$ fit into a long exact sequence
$$
\dots\to H^k(U) \to H^k(\Cal H^*|_U,D)\to H^{k-1}(U) \to
H^{k+1}(U)\to\dots 
$$ 
In this sequence, the connecting homomorphism $H^{k-1}(U) \to
H^{k+1}(U)$ is given by the wedge product with the cohomology class of
$\om|_U$, where $\om$ is the symplectic form on $M$ from Proposition
\ref{2.2}. 

In particular, for contractible $U$, the cohomology groups of $(\Cal
H^*,D)$ are isomorphic to $\mathbb R$ in degrees $0$ and $1$, while
all higher cohomologies vanish. On the other hand, if $\om$ is not
exact on $M$, then $H^k(\Cal H^*,D)\cong H^k(M)$ for $k=0,1$.  
\end{thm*}
\begin{proof}
By Corollary \ref{2.3}, the cohomology of $(\Cal H^*|_U,D)$ coincides
with the cohomology of $(\Cal A^*(q^{-1}(U)),d)$. Now on
$U_\#:=q^{-1}(U)$, let us denote by $i_\xi$ the insertion operator
defined by the vector field $\xi$. Applying Proposition \ref{2.3} to
$q|_{U_\#}:U_\#\to U$, we get a short exact sequence of complexes
$$
0\to (\Om^*(U),d)\to (\Cal A^*(q^{-1}(U)),d)
\overset{i_\xi}{\longrightarrow} (\Om^{*-1}(U),d)\to 0 .
$$
The corresponding long exact sequence in cohomology has the form
claimed in the theorem. To describe the connecting homomorphism,
consider a closed $(k-1)$--form $\ph$ on $U$. As a preimage in $\Cal
A^k(q^{-1}(U))$ under $i_\xi$, we can use $\al\wedge q^*\ph$. The
exterior derivative of this is $d\al\wedge q^*\ph=q^*(\om\wedge\ph)$
by Proposition \ref{2.1}. 

In particular, if $\om$ is exact, then all connecting homomorphisms
vanish, and one obtains a short exact sequence $0\to H^k(U)\to H^k(\Cal
H^*|_U,D)\to H^{k-1}(U)\to 0$ for each $k$. From this the claim on
contractible $U$ follows. 

Next, the long exact sequence immediately implies that $H^0(U)\cong
H^0(\Cal H^*|_U,D)$. If $\om$ is non--exact on $M$, then wedging with
$\om$ is an injection $H^0(M)\to H^2(M)$, so the long exact sequence
shows that $H^1(M)\cong H^1(\Cal H^*,D)$.
\end{proof}

\begin{remark*}
Of course, in the case of non--exact $\om$, one can get more out of
the long exact sequence than just the final claim of the theorem. We
have emphasized that last result only, because it is (for a much more
general family of complexes, but only in the case of the simply
connected space $\Bbb CP^n$) exactly what is need to prove the
integral geometry results of \cite{EG}.

A natural context to further exploit the long exact sequence in the
theorem would for example be the case of symplectic manifolds which
satisfy the hard Lefschetz theorem, like K\"ahler manifolds. In this
case, the wedge product by the cohomology class of $\om$ is injective
below half the dimension and surjective above, and the cohomology of
$(\Cal H^*,D)$ nicely relates to primitive cohomology. 
\end{remark*}

\section{The Rumin--Seshadri complex on general conformally symplectic
  manifolds}\label{3}

To conclude this article, we want to show how the quotient
construction from section \ref{2} can be used to obtain a complex on
general cs--manifolds and prove that this complex is intrinsic to the
conformally symplectic structure. Moreover, we again obtain an exact
sequence, which nicely relates its cohomology to de--Rham cohomology.

\subsection{Local contactifications}\label{3.1}
As indicated in section \ref{2} already, we view a cs--manifold as a
manifold $M$ of even dimension $2n\geq 4$ together with a line
subbundle $\ell\subset \La^2T^*M$ such that all non--zero elements of
$\ell$ are (point--wise) non--degenerate as bilinear forms and such
that $\ell$ admits local non--vanishing sections which are closed as
two--forms on $M$. Observe that if $\Om$ is a local closed
non--vanishing section of $\ell$ then any other local section of
$\ell$ is of the form $f\Om$ for a smooth function $f$. But then
$d(f\Om)=df\wedge\Om$ so since $\dim(M)>2$, this vanishes only if
$df=0$ by non--degeneracy of $\Om$. Hence these local non--vanishing
closed sections are determined up to a constant factor.

Given such a structure, each point $x\in M$ has an open
neighborhood $U$ in $M$ such that there are local sections of $\ell|_U$
which are exact as two forms on $M$. Then it is well known that such
neighborhoods can be realized as cs--quotients:

\begin{lemma*}
Let $(M,\ell)$ be a conformally symplectic manifold and let $U\subset
M$ be an open subset and let $\be\in\Om^1(U)$ be such that $d\be$ is a
section of $\ell$ which is nowhere vanishing on $U$. Put
$U_\#:=U\x\Bbb R$ and let $t$ be the standard coordinate on the second
factor.

Then $\al:=dt+q^*\be$ is a contact form on $U_\#$ such that
$q:=\pr_1:U_\#\to U$ is a cs--quotient with respect to the Reeb field
$\xi:=\partial_t$ of $\al$.
\end{lemma*}
\begin{proof}
We have $d\al=q^*d\be$ so this is non--degenerate on $TU\subset
TU_\#$, so clearly $dt\wedge (q^*d\be)^n=\al\wedge (d\al)^n$ is a
volume form on $U_\#$. Thus $\al$ is a contact form, and it is obvious
that its Reeb field $\xi$ equals $\partial_t$. This shows that $\xi$
spans the vertical subbundle of $q=\pr_1$ and by construction $d\al$
descends to the section $d\be$ of $\ell$.  
\end{proof}

There also is a local uniqueness result:

\begin{prop*}
Let $q:M_\#\to M$ and $\tilde q:\tilde M_\#\to \tilde M$ be
cs--quotients with respect to infinitesimal automorphisms $\xi\in\frak
X(M_\#)$ and $\tilde\xi\in\frak X(\tilde M_\#)$. Suppose that
$\ph:\tilde M\to M$ is a cs--diffeomorphism,
i.e.~$\ph^*(\ell)=\tilde\ell$.

Then for each point $\tilde x\in \tilde M$ and each point $\tilde
u\in\tilde M_\#$ with $\tilde q(\tilde u)=\tilde x$, there are open
neighborhoods $\tilde U$ of $\tilde x$ in $\tilde M$ and $\tilde U_\#$
of $\tilde u$ in $\tilde M_\#$ such that $\tilde q$ restricts to a
surjective submersion $\tilde U_\#\to\tilde U$ with connected fibers
and an immersion $\Ph:\tilde U_\#\to M_\#$ which restricts to a
contactomorphism onto its image and satisfies $q\o\Ph=\ph\o\tilde q$
and $\Ph^*\xi=\la\tilde \xi$ for some non--zero constant $\la$.
\end{prop*}
\begin{proof}
Given $\tilde x$, choose a contractible neighborhood $U$ of
$x=\ph(\tilde x)$ in $M$ such that there is $\be\in\Om^1(U)$ for which
$d\be$ is a section of $\ell$, which is nowhere vanishing on
$U$. Further, take any point $u\in q^{-1}(x)$ and let
$\al\in\Om^1(M_\#)$ be the unique contact form such that
$\al(\xi)=1$. Then by part (2) of Proposition \ref{2.2} $d\al$
descends to a locally non--vanishing closed section of
$\ell$. Replacing $\be$ by a constant multiple if necessary, we may
thus assume that $d\al$ descends to $d\be$. Since we also know that
$i_\xi d\al=i_\xi dq^*\be=0$, we conclude that $d(\al-q^*\be)=0$.

Possibly shrinking $U$, we can choose an open neighborhood $U_\#$ of
$u$ in $M_\#$ such that $q:U_\#\to U$ is a submersion with connected
fibers and such that $\al-q^*\be=df$ for some smooth $f:U_\#\to\Bbb R$
with $f(u)=0$. By construction $df(\xi)=1$, so $\ker(df)$ is always
transversal to $\ker(Tq)=\Bbb R\cdot\xi$. This implies that
$(q,f):U_\#\to U\x\Bbb R$ is a local diffeomorphism, so possibly again
shrinking $U$ and $U_\#$ we may assume that it is a diffeomorphism
$U_\#\to U\x (-\ep,\ep)$ for some $\ep>0$. 

Now we apply the same construction to the neighborhood $\ph^{-1}(U)$
of $\tilde x$ in $\tilde M$, the form $\tilde\be=\ph^*\be\in\Om^1(U)$,
and the point $\tilde u\in\tilde q^{-1}(\tilde x)$. Note however, that
we are not allowed to rescale $\tilde\be$ any more, so we can only
obtain $\la\tilde\al=\tilde q^*\tilde\be+d\tilde f$ for some non--zero
constant $\la$. Again we can arrange things in such a way that
$(\tilde q,\tilde f)$ is a diffeomorphism $\tilde U_\#\to \tilde U\x
(-\ep,\ep)$. 

Now define $\Ph:\tilde U_\#\to U_\#$ as $(q,f)^{-1}\o (\ph,\id)\o
(\tilde q,\tilde f)$. Then of course $\Ph$ is a diffeomorphism such
that $\Ph^*df=d\tilde f$ and $q\o\Ph=\ph\o\tilde q$ and thus
$\Ph^*q^*\be=\tilde q^*\tilde\be$. This shows that
$\Ph^*\al=\la\tilde\al$, so $\Ph$ is a contactomorphism and the Reeb
field $\xi$ of $\al$ is pulled back to the Reeb field of $\la\tilde\al$,
which is $\frac1\la\tilde\xi$. 
\end{proof}

\subsection{An intrinsic complex via local push downs}\label{3.2} 
The local existence and uniqueness results from \ref{3.1} are not
quite enough to construct an intrinsic sequence of differential
operators from Corollary \ref{2.3}. The point is that in the
uniqueness result Proposition \ref{3.1}, we do not get true
compatibility between the infinitesimal automorphism, but only
compatibility up to a non--zero constant multiple. Looking at the
first four isomorphisms in Proposition \ref{2.3}, which are used to
construct the descended complex, we see that the first two of them
remain unchanged if $\xi$ is rescaled by a non--zero constant. The other two
isomorphisms, however, involve the contact form $\al$ characterized by
$\al(\xi)=1$, so these change if $\xi$ changes.

We can easily correct this by using a slightly modified version of the
construction from \ref{2.3}. Consider the line bundle
$\ell\subset\La^2T^*M$ defining the conformally symplectic
structure. This can be viewed as an abstract line bundle on $M$, and
then we can consider the spaces $\Om^k(M,\ell)$ of $\ell$--valued
differential forms on $M$. Moreover, we have seen above that $\ell$
admits local sections which are closed as two--forms on $M$, and these
are uniquely determined up to a constant multiple. Thus we can define
a linear connection $\nabla$ on $\ell$ by requiring that these
sections are parallel for $\nabla$, which of course implies that
$\nabla$ is flat. In particular, twisting the exterior derivative by
$\nabla$ we obtain the twisted de--Rham complex
$(\Om^*(M,\ell),d^{\nabla})$.

Now suppose that $q:M_\#\to M$ is a cs--quotient with respect to a
transversal infinitesimal automorphism $\xi$ of $M_\#$. Then from the
proof of Proposition \ref{2.2} we see that if $\al$ is a contact form
on $M$ such that $\al(\xi)$ is constant, then $d\al$ descends to a
non--vanishing section of $\ell$. Any nonzero element in (a fiber of)
$\ell$ can be obtained in this way, so we conclude that $\xi$ gives
rise to a section $\si_\xi$ of the dual bundle $\ell^*\to M$ and thus
also to a (tensorial) map $\Om^k(M,\ell)\to\Om^k(M)$ which we denote
by the same symbol. Now we can modify Proposition \ref{2.3} and
construct an isomorphism between $\Om^k(M,\ell)$ and
$\ker(L_\xi)\subset\Ga(\La^kH^*\otimes Q^*)$ by mapping $\ph$ to
$\al\wedge q^*(\si_\xi(\ph))$, where again $\al$ is the unique contact
form such that $\al(\xi)=1$. Evidently, if we multiply $\xi$ by a
non--zero constant, then $\si_\xi$ gets multiplied by the same
constant while $\al$ gets divided by that constant, so the resulting
isomorphism stays the same. 

There is also a natural analog $\Om_0^k(M,\ell)$ of $\Om_0^k(M)$ for
$k\geq n$. In \ref{2.2}, we have introduced the canonical section
$\mathbf{\Om}$ of $\La^2T^*M\otimes\ell^*$. Wedging with
$\mathbf{\Om}$ can also be interpreted as a bundle map
$\La^kT^*M\otimes\ell\to\La^{k+2}T^*M$ for each $k\geq 0$, which is
surjective for $k\geq n-1$. For $k\geq n$, we denote by
$(\La^kT^*M\otimes\ell)_0$ the kernel of this bundle map and by
$\Om^k_0(M,\ell)$ the space of sections of this bundle. As in
Proposition \ref{2.3} it is clear that the isomorphism between
$\Om^k(M,\ell)$ and $\ker(L_\xi)\subset\Ga(\La^kH^*\otimes Q^*)$ from
above restricts to an isomorphism between $\Om^k_0(M,\ell)$ and the
kernel of $L_\xi$ in the subspace $\Ga(\La^k_0H^*\otimes Q^*)$ for
$k\geq n$.

Having all that in hand, we can construct a complex intrinsic to a
conformally symplectic structure:
\begin{thm*}
Let $(M,\ell)$ be a cs--manifold of dimension $2n$. Then there is a
differential complex $(\Cal H^*,D)$ which is intrinsically associated
to the cs--structure such that $\Cal H^i=\Om^i_0(M)$ for $i\leq n$
and $\Cal H^i=\Om^{i-1}_0(M,\ell)$ for $i>n$. The operators $D$ are
all of first order except for $D:\Om^n_0(M)\to\Om^n_0(M,\ell)$, which
has order two. 
\end{thm*}
\begin{proof}
Given $\ph\in\Cal H^i$ we consider the restriction $\ph|_U$ for an
open subset $U\subset M$ over which we find a local contactification
$q:U_\#\to U$. Over $U$ we can use the push down construction from
\ref{2.3} to obtain a section $D(\ph|_U)\in\Cal
H^{i+1}|_U$. Proposition \ref{3.1} and naturality of the Rumin complex
under contactomorphisms imply that this section is independent of the
choice of contactification and that the resulting local sections piece
together to a well defined element $D(\ph)\in\Cal H^{i+1}$.  By
construction, this defines a complex which is intrinsic to the
cs--structure.
\end{proof}

Of course, if we start with a symplectic structure, we can use the
symplectic form to trivialize the bundle $\ell$ and thus identity
$(\Om^*(M,\ell),d^{\nabla})$ with $(\Om^*(M),d)$. It is then easy to
check directly that the complex constructed here coincides with the
one from the last section of \cite{BEGN}, where it is called the
\textit{Rumin--Seshadri complex}, and we keep this name in the more
general setting of lcs--structures. That reference also contains some
information on the history of this complex and discusses the relations
of this complex to \cite{TY1,TY2}. In the situation of a general
conformally symplectic structure, one can fix a section of $\ell$ to
get to the setting of \cite{Eastwood}, and under the resulting
identification we obtain the complex constructed there. 

\subsection{Cohomology of the Rumin--Seshadri complex on
  conformally symplectic structures}\label{3.3}

To conclude our article, we show how the push--down methods can be
used to analyze the cohomology of the Rumin--Seshadri complex on
conformally symplectic manifolds. The result is essentially contained
in Theorem 2 of \cite{Eastwood}, apart from the fact that we do not
work with a fixed section of $\ell$. Our proof via the quotient
constuction is different, however, and we include it for completeness.

Consider the (tensorial) map $\Om^k(M,\ell)\to\Om^{k+2}(M)$ defined by
wedging with the canonical section
$\mathbf{\Om}\in\Om^2(M,\ell^*)$. If $\si_0$ is a locally
non--vanishing section of $\ell$ which is parallel for $\nabla$, then
by definition $d(\mathbf{\Om}\wedge\si_0)=0$, while for
$\ph\in\Om^k(M)$ we have
$d^{\nabla}(\ph\otimes\si_0)=d\ph\otimes\si_0$. Using these facts, one
immediately verifies that $d(\mathbf{\Om}\wedge\ps)=\mathbf{\Om}\wedge
d^{\nabla}\ps$ holds for all $\ps\in\Om^k(M,\ell)$. Otherwise put
$d^{\nabla^*}\mathbf{\Om}=0$, where $\nabla^*$ is the connection on
$\ell^*$ induced by $\nabla$. In particular, we can form the class
$[\mathbf{\Om}]$ in the twisted de--Rham cohomology group
$H^2(M,\ell^*)$ and wedging with this class defines a map
$H^k(M,\ell)\to H^{k+2}(M)$. Using this, we can formulate:

\begin{thm*}
Let $(M,\ell)$ be a conformally symplectic manifold. Then the
cohomology groups of the Rumin--Seshadri complex $(\Cal H^*,D)$ fit
into a long exact sequence of the form
$$
\dots \to H^k(M)\to H^k(\Cal H^*,D)\to H^{k-1}(M,\ell)\to
H^{k+1}(M)\to\dots,
$$
in which the connecting homomorphism $H^{k-1}(M,\ell)\to
H^{k+1}(M)$ is given by the wedge product with $[\mathbf{\Om}]\in
H^2(M,\ell^*)$. 
\end{thm*}
\begin{proof}
We continue modifying the results from Proposition \ref{2.3} as started
in \ref{3.2}. Given an open subset $U\subset M$ over which there is a
contactification, we can define an isomorphism
$\Om^k(U)\oplus\Om^{k-1}(U,\ell)\to\Cal A^k(q^{-1}(U))$ by
$(\ph,\ps)\mapsto q^*\ph+\al\wedge\si_\xi(\ps)$ (with notation as in
\ref{3.2}). This remains unchanged if $\xi$ is replaced by a constant
multiple. One immediately verifies that under this isomorphism, the
action of the exterior derivative on $\Cal A^k(q^{-1}(U))$ corresponds
to $(\ph,\ps)\mapsto (d\ph+\mathbf{\Om}\wedge\ps,-d^{\nabla}\ps)$, and
this defines a differential $\tilde D$ on
$\Om^*(U)\oplus\Om^{*-1}(U,\ell)$.

Combining this with the constructions of \ref{2.3} we get a filtration
of the complex $(\Om^*(U)\oplus\Om^{*-1}(U,\ell),\tilde D)$ such that
the associated spectral sequence produces the complex $(\Cal
H^*|_U,D)$, which thus computes the same cohomology. By naturality of
all the involved constructions, these fit together to define a
filtration on 
$$
(\Om^*(M)\oplus\Om^{*-1}(M,\ell),\tilde D)
$$ 
such that the associated spectral sequence produces $(\Cal H^*,D)$,
which thus computes the same cohomology. From this the result follows,
since there is an evident short exact sequence
$$
0\to (\Om^*(M),d)\to (\Om^*(M)\oplus\Om^{*-1}(M,\ell),\tilde D)\to
(\Om^{*-1}(M,\ell),-d^\nabla)\to 0
$$ for which the connecting homomorphism in the resulting long exact
sequence has the claimed form.
\end{proof}

\end{document}